\documentclass{amsart}
\usepackage{eurosym}

\usepackage{amsmath}
\usepackage{amsfonts}
\usepackage[hidelinks]{hyperref}
\usepackage{graphicx}
\usepackage{cite}
\usepackage{float}
\usepackage{color}
\usepackage{geometry}
\newtheorem{theorem}{Theorem}
\theoremstyle{plain}

\newtheorem{definition}{Definition}

\newtheorem{lemma}{Lemma}

\newtheorem{remark}{Remark}

\numberwithin{equation}{section}
\begin{document}
\title{Approximation by Steklov Neural Network Operators}
\author{Şerife Nur Karaman}
\address{Selcuk University, Faculty of Science, Department of Mathematics,
	Selcuklu, 42003, Konya, Turkey}
\email{serifenurkaraman@yahoo.com}
\author{Metin Turgay}
\address{Selcuk University, Faculty of Science, Department of Mathematics,
	Selcuklu, 42003, Konya, Turkey}
\email{metinturgay@yahoo.com}
\author{Tuncer Acar}
\address{Selcuk University, Faculty of Science, Department of Mathematics,
	Selcuklu, 42003, Konya, Turkey}
\email{tunceracar@ymail.com}
\maketitle

\begin{abstract}
	The present paper deals with construction of newly family of Neural Network operators, that is, Steklov Neural Network operators. By using Steklov type integral, we introduce a new version of Neural Network operators and we obtain some convergence theorems for the family, such as, pointwise and uniform convergence, rate of convergence via modulus of continuity.
\end{abstract}

\section{Introduction}
Sampling type series were introduced to address the limitations of the well-established Whittaker-Kotel'nikov-Shannon (WKS) theorem (see \cite{Whittaker1915, Kotelkinov1933, Shannon1949}), which is applicable only to signals that are both band-limited and possess finite energy. These constraints, imposed by the Paley-Wiener theorem, restrict the theorem's use to highly regular functions. To broaden its applicability P. L. Butzer and his collaborators introduced generalized sampling series \cite{Butzer1977, Butzer1986, Butzer1993, Butzer1998}. 

The generalized sampling series is defined by
$$
S_w^\chi f(x) = \sum_{k \in \mathbb{Z}} f\left(\frac{k}{w}\right) \chi(wx - k), \quad x \in \mathbb{R}, w > 0,
$$
where $\chi$ represents a kernel function that meets some approximation criterias. These series are designed to reconstruct functions using discrete sample values, and they have been applied to spaces of continuous functions. 

An extension of this approach, known as sampling Kantorovich operators, was later introduced to accommodate not only continuous but also integrable functions. These operators modify the generalized sampling formula by replacing sample values with mean values of the function over intervals $[k/w,(k+1)/w]$, thus we obtain a more flexible reconstruction method for broader classes of functions, e.g. $L^{1}$ spaces, see for details \cite{Costarelli2019a, Bardaro2007, Angeloni2021}.

Several types of sampling-type operators have been extensively studied in the literature. These operators, explored in both theoretical and applied contexts, are central to various fields. Generalized sampling operators (see, e.g., [6, 14]) form the basis for reconstructing continuous signals from sample values, while the Kantorovich modifications (see, e.g., [1, 8, 19, 26, 37, 41]) extend this concept by incorporating integration over intervals. Similarly, generalized Durrmeyer sampling operators (see, \cite{Bardaro2014, Costarelli2023}) provide a further generalization by integrating polynomial terms. Additional developments, such as exponential sampling series and modifications introduced by Bardaro et al. \cite{Bardaro2017} (for further see, \cite{Acar2023a, Kursun2021, Acar2023}), have broadened the applicability of these methods. These advances have enabled the study of phenomena like light scattering and diffraction. As research continues, new applications and modifications are emerging, as seen in works covering weighted approximations by sampling-type operators (see, e.g., \cite{Acar2022, Aral2022, Kursun2023, Alagoz2022, Ozer2023}).

In line with the ongoing research in sampling theory, Costarelli \cite{Costarelli2024} has introduced a novel class of operators known as Steklov sampling operators, denoted by $S_r^w$. These operators are constructed using a kernel function $\chi : \mathbb{R} \rightarrow \mathbb{R}$, which acts as a discrete approximate identity. The core idea behind of this approach is to reconstruct a signal $f$ by utilizing a set of sample values, specifically Steklov integrals of order $r$, evaluated at the nodes $k/w$, where $k \in \mathbb{Z}$ and $w > 0$. Steklov integrals, commonly employed in constructive approximation theory, provide a way to generate smooth approximations for functions that may lack regularity, further extending the flexibility and applicability of sampling-type methods.

In recent years, neural network (NN) operators have gained significant attention across various fields, including approximation theory \cite{Cao2009, Cao2016, Fard2016}, artificial intelligence and neuroscience \cite{Austin2016, Baldi2016,Ito2001, Llanas2006, Maiorov2000, Maiorov2006, Makovoz1998, Mhaskar2016}. These operators, which stem from the classical theory of artificial neural networks (ANNs), provide constructive methods for approximating functions. Neural networks are typically modelled mathematically as finite linear combinations of activation functions, which involve the scalar product of multivariate variables with weights, adjusted by a bias term. The activation functions commonly used are sigmoidal due to their biological relevance, as they simulate the activated and non-activated states of neurons.

The study of NN operators has led to the development of several versions, each with unique properties. For instance, the classical version focuses on pointwise and uniform approximation of continuous functions, while the max-product version offers sharper approximations with higher accuracy. Kantorovich-type NN operators, on the other hand, provide an $L^1$ extension, making them suitable for broader classes of functions.

Despite the qualitative exploration of NN operators, including convergence and approximation order, the quantitative aspect remains an area of ongoing research. The quantitative estimates, especially concerning the approximation capabilities of NN-type algorithms, are crucial for practical applications, where neural networks are used to model data from training sets. These estimates depend heavily on the choice of target functions, often sigmoidal, driven by biological motivations.

Overall, NN operators play a vital role in both theoretical studies and practical applications, offering flexible tools for function approximation and prediction in various domains.

In this paper, by motivating the Costarelli's work \cite{Costarelli2024}, we introduce Steklov Neural Network Operators $\left( SNNO\right)$
\begin{equation}
		F_{n}^{r}\left( f;x\right) :=\frac{\underset{k=\left\lceil
			na\right\rceil }{\overset{\left\lfloor nb\right\rfloor-r }{\sum }}%
		f_{r,n}\left( \frac{k}{n}\right) \phi _{\sigma }\left( nx-k\right) }{%
		\underset{k=\left\lceil na\right\rceil }{\overset{\left\lfloor
				nb\right\rfloor-r }{\sum }}\phi _{\sigma }\left( nx-k\right) }, \ \ \ x\in \left[ a,b\right], \ \  n\in \mathbb{N}
\end{equation}
for any $f:\left[ a,b\right] \rightarrow \mathbb{R}$ bounded function, where $f_{r,n}$ denotes the denote the Steklov-type integrals and $\phi _{\sigma }$ density function generated by sigmoidal function $\sigma $.

First of all, in Section \ref{sec:preliminaries} we introduce concepts that we will use in later chapters in the preliminary section. In Section \ref{sec:main-results},  we start with by obtaining well-definiteness of the Steklov neural network operators, then we present pointwise and uniform convergence of the operators in the spaces of continuous functions. Also, we give rate of convergence of the newly constructed family of operators via modulus of continuity.

\section{Preliminaries}\label{sec:preliminaries}
In this section, we establish some preliminary results that will be useful in the rest of the paper. Let $I=\left[ a,b\right]$ be an interval. We define $C(\textit{I})$ the space all of functions $f:\left[ a,b\right] \rightarrow \mathbb{R}$, which are continuous on $\left[a,b\right]$ and endowed by usual sup norm. 

For any function $\phi :\mathbb{R}\rightarrow \mathbb{R}$ and $\beta \in \mathbb{N}$, we define the truncated algebraic moment of order $\beta$ by
\begin{equation*}
	m_{\beta }^{n}\left( \phi ,u\right) :=\sum_{k=-n}^{n}\phi \left( u-k\right)
	\left( u-k\right) ^{\beta }, \quad u\in \mathbb{R}
\end{equation*}
for every $n\in \mathbb{N}^{+}$. Morever, we also define the discrete absolute moments of $\phi $ by
\begin{equation*}
	M_{\beta }\left( \phi \right) :=\sup_{u\in \mathbb{R}}\sum_{k\in \mathbb{Z}}\left\vert \phi \left( u-k\right) \right\vert \left\vert u-k\right\vert
	^{\beta },
\end{equation*} 
for $\beta\geq 0$.

Now, let us recall the definition of density functions used to define neural network operators. A function $\sigma :\mathbb{R}\rightarrow \mathbb{R}$ is called a sigmoidal function if and only if
\begin{equation*}
	\lim_{x\rightarrow -\infty }\sigma \left( x\right) =0\text{ \ and \ }%
	\lim_{x\rightarrow \infty }\sigma \left( x\right) =1.
\end{equation*}
In this work, we consider non-decreasing sigmoidal functions $\sigma$, such that $\sigma\left(r+2\right)>\sigma\left(r\right)$ (this condition is merely technical), and such that satisfying the
following assumptions:
\begin{itemize}
\item[$\left( S1\right) $] $\sigma \left( x\right) -1/2$ is an odd function;

\item[$\left( S2\right) $] $\sigma \in C^{2}\left( \mathbb{R}\right) $ is concave for $x\geq 0$;

\item[$\left( S3\right) $] $\sigma \left( x\right) =O\left( \left\vert
x\right\vert ^{-1-\alpha }\right) $ as $x\rightarrow -\infty $ for some $%
\alpha >0$,
\end{itemize}
according to the general theory. We recall the definiton of the density function $\phi _{\sigma }$ generated by $\sigma $:
\begin{equation*}
\phi _{\sigma }\left( x\right) \colon =\frac{1}{2}\left[ \sigma \left(
x+1\right) -\sigma \left( x-1\right) \right] ,\ \\ \ x\in \mathbb{R}
\end{equation*}
In the following Lemma, we state some important properties of $\phi _{\sigma }$. We omit the proof of the following Lemma, since it could be obtain by straightforward computation with the same process given in \cite{Costarelli2013}.
\begin{lemma}\label{lemma-phi-prop}
Let $n\in \mathbb{N}^{+}.$ Then,
\begin{enumerate}
\item[$\left(i\right) $] $\phi _{\sigma }\left( x\right) \geq 0$  for every $x\in \mathbb{R}$ and in particular $\phi _{\sigma }\left( r + 1\right) >0,$

\item[$\left(ii\right) $] lim$_{x\rightarrow \pm \infty }\phi _{\sigma}\left( x\right) =0$,

\item[$\left(iii\right) $] $\phi _{\sigma }$ is an even function,

\item[$\left(iv\right) $] For every $x\in \mathbb{R}$, we have $\sum_{k\in\mathbb{Z}}\phi _{\sigma }\left( x-k\right) =1,$

\item[$\left(v\right) $] $\phi _{\sigma }\left( x\right) $ is non-decreasing function for $x<0$, and non-increasing for $x\geq 0,$

\item[$\left(vi\right) $] Let $\alpha $ a positive constant of condition $\left( S3\right) $. Then,
\begin{equation*}
\phi _{\sigma }\left( x\right) =\mathcal{O}\left( \left\vert x\right\vert ^{-1-\alpha }\right) ,\ \ as\ \ \ \ x\rightarrow \pm \infty \text{.}
\end{equation*}

\item[$\left(vii\right) $] Let $x\in \left[ a,b\right] $. Then,
\begin{equation*} 
	\sum_{k=\left\lceil na\right\rceil}^{{\left\lfloor nb\right\rfloor -r }}\phi _{\sigma }\left( nx-k\right) \geq \phi _{\sigma }\left( r + 1\right) >0\text{.}
\end{equation*}

\item[$\left(viii\right) $] For every $\gamma>0$, we have 
\begin{equation*}
	\lim_{n\rightarrow \infty } \sum_{\left|x-k\right|>\gamma n} \phi_{\sigma}\left(x-k\right) = 0,
\end{equation*}
uniformly with respect to $x\in\mathbb{R}$.

\end{enumerate}
\end{lemma}

Let $f:\left[ a,b\right] \rightarrow \mathbb{R}$ is bounded function and $n\in \mathbb{N}^{+}$ such that ${\lceil n a\rceil}\leq{\lfloor n b\rfloor}.$ The positive linear neural network operators activated by sigmoidal function $\sigma$ are defined as 
\begin{align*}
	F_n(f ; x): & =\frac{\sum\limits_{k=\lceil n a\rceil}^{\lfloor n b\rfloor} f\left(\frac{k}{n}\right) \phi_\sigma(n x-k)}{\sum\limits_{k=\lceil n a\rceil}^{\lfloor n b\rfloor} \phi_\sigma(n x-k)}, \ \ x\in \left[ a,b\right]
\end{align*}
(see \cite{Costarelli2013}).
 
In approximation theory, it is well-known that Steklov-type integrals provide a powerful method for achieving regular approximations of functions that may not exhibit regularity. Various definitions of these integral means have been explored in the literature. These integral means are widely utilized for smoothing and approximating functions, offering a flexible tool for addressing irregularities in the approximation process. In this context, we consider a particular version of the Steklov integrals, inspired by the formulation introduced by Papov and Sendov \cite{Papov1983}:
\begin{equation*}
f_{r,h}\left( x\right) :=\left( -h\right) ^{-r}\overset{h}{\underset{0} {\int }}\cdots\overset{h}{\underset{0}{\int }}\sum_{m=1}^{r}\left( -1\right) ^{r-m+1}\binom{r}{m}f\left( x+\frac{m}{r}\left( t_{1}+t_{2}+\cdots+t_{r}\right) \right) dt_{1}\cdots dt_{r}
\end{equation*}
for any locally integrable function of the form $f:\mathbb{R}\rightarrow \mathbb{R}$ with $r\in \mathbb{N}^{+}$, $h>0$.

\begin{definition}
	Let $r,n\in \mathbb{N}^{+}$ such that ${\lceil n a\rceil}\leq{\lfloor n b\rfloor-r}.$ We define the Steklov neural network operators of order r $\left( SNNO\right) _{r}$ as
	\begin{align*}
		F_n^r(f ; x):=&\frac{\sum\limits_{k=\lceil n a\rceil}^{\lfloor n b\rfloor-r} f_{r, \frac{1}{n}}\left(\frac{k}{n}\right) \phi_\sigma(n x-k)}{\sum\limits_{k=\lceil n a\rceil}^{\lfloor n b\rfloor-r} \phi_\sigma(n x-k)} \\
		=&\frac{\sum\limits_{k=\lceil n a\rceil}^{\lfloor n b\rfloor-r}\left[n^r \int\limits_{0}^{\frac{1}{n}} \cdots \int\limits_{0}^{\frac{1}{n}} \sum\limits_{m=1}^r(-1)^{1-m}{r\choose m} f\left(\frac{k}{n}+\frac{m}{r}\left(t_1+t_2+\cdots+t_r\right)\right) d t_1 \cdots d t_r\right] \phi_\sigma(n x-k)}{\sum\limits_{k=\lceil n a\rceil}^{\lfloor n b\rfloor-r} \phi_\sigma(n x-k)},
	\end{align*}
$ x\in [a,b]$, where $f:\left[ a,b\right] \rightarrow \mathbb{R}$ bounded function, $f_{r,n}$ denotes the denote the Steklov-type integrals and $\phi _{\sigma }$ density function generated by sigmoidal function $\sigma $.
\end{definition}
\begin{remark}
	In case $r=1,$ it is easy to see that we have classical neural network operators of the Kantorovich type, see \cite{Costarelli2014}.
\end{remark}
\section{Main Results}\label{sec:main-results}

In this section, we present main results of newly constructed family of Steklov type neural network operators. First of all, we give the well-definiteness of the operators $F_n^{r}$.
\begin{theorem}
For every $r, n\in \mathbb{N}^{+}$ such that ${\lceil n a\rceil}\leq{\lfloor n b\rfloor-r}$ and $f:\left[ a,b\right] \rightarrow \mathbb{R}$ is bounded. Steklov Neural Network Operators are well defined.
\end{theorem}

\begin{proof}
By Lemma \ref{lemma-phi-prop} (i) and Lemma \ref{lemma-phi-prop} (vii), we easily have
\begin{align*}
	& \left|F_n^r(f ; x)\right|\\
	=&\left|\frac{\sum\limits_{k=\lceil n a\rceil}^{\lfloor n b\rfloor-r}\left[n^r \int\limits_0^{1 / n} \cdots \int\limits_0^{1 / n} \sum\limits_{m=1}^r(-1)^{1-m}\binom{r}{m} f\left(\frac{k}{n}+\frac{m}{r}\left(t_1+t_2+\cdots+t_r\right)\right) d t_1 \cdots d t_r\right] \phi_\sigma(n x-k)}{\sum\limits_{k=\lceil n a\rceil}^{\lfloor n b\rfloor-r} \phi_\sigma(n x-k)}\right| \\
	\leq& \frac{1}{\phi_\sigma(1)} \sum_{k=\lceil n a\rceil}^{\lfloor n b\rfloor-r}\left|\left[n^r \int_0^{1 / n} \cdots \int_0^{1 / n} \sum_{m=1}^r(-1)^{1-m}\binom{r}{m} f\left(\frac{k}{n}+\frac{m}{r}\left(t_1+t_2+\cdots+t_r\right)\right) d t_1 \cdots d t_r\right]\right| \phi_\sigma(n x-k) \\
	\leq& \frac{\left(2^r-1\right)\|f\|_{\infty}}{\phi_\sigma(1)} \sum_{k \in \mathbb{Z}} \phi_\sigma(n x-k) \\
	\leq& \frac{\left(2^r-1\right)\|f\|_{\infty}}{\phi_\sigma(1)}<+\infty,
\end{align*}
which means that $\left( F_{n}^{r}f\right)$ are bounded.
\end{proof}

Now, we present pointwise and uniform convergence of the family of operators $\left( F_{n}^{r}\right)$.

\begin{theorem}
Let $r, n\in \mathbb{N}^{+}$ such that ${\lceil n a\rceil}\leq{\lfloor n b\rfloor-r}$ and $f:\left[ a,b\right] \rightarrow \mathbb{R}$ be a bounded function. Then%
\begin{equation*}
\lim_{n\rightarrow \infty }\left( F_{n}^{r}f\right) \left( x\right) =f\left(
x\right)
\end{equation*}
at any continuity point $x\in \left[ a,b\right] $ of the function $f$. If 
$f\in C\left( \left[ a,b\right] \right) $, we have
\begin{equation*}
	\lim_{n\rightarrow \infty }\left\Vert F_{n}^{r}f-f\right\Vert _{\infty }=0%
	\text{.}
\end{equation*}
\end{theorem}

\begin{proof}
Let $x\in \left[ a,b\right] $ be a fixed point of continuity of $f$.
For $\varepsilon >0$ there exists $\delta >0$ such that for every $y\in %
\left[ x-\delta ,x+\delta \right] \cap \left[ a,b\right] $, $\left\vert
f\left( y\right) -f\left( x\right) \right\vert <\varepsilon $.
\begin{align*}
	&\left( F_{n}^{r}f\right) \left( x\right) - f\left( x\right)  \\
	=&\frac{		\sum\limits_{k=\left\lceil na\right\rceil }^{\left\lfloor			nb\right\rfloor-r}\left[ n^{r} \int\limits_{0}^{1/n}\cdots		\int\limits_{0}^{1/n}\sum\limits_{m=1}^{r}\left(-1\right) ^{1-m}\binom{r}{m}f\left( \frac{k}{n}+\frac{m}{r}\left(t_{1}+t_{2}+\cdots+t_{r}\right) \right) dt_{1}\cdots dt_{r}\right] \phi _{\sigma}\left( nx-k\right) }{\sum\limits_{k=\left\lceil na\right\rceil }^{\left\lfloor nb\right\rfloor-r}\phi _{\sigma }\left( nx-k\right) } -f\left( x\right) 
\end{align*}
\begin{align*}
	= &\frac{\sum\limits_{k=\left\lceil na\right\rceil}^{\left\lfloor			nb\right\rfloor-r}\left[ n^{r}\int\limits_{0}^{1/n}\cdots	\int\limits_{0}^{1/n}\left[ \sum\limits_{m=1}^{r}	\left( -1\right) ^{1-m}\binom{r}{m}f\left( \frac{k}{n}+\frac{m}{r}\left( t_{1}+t_{2}+\cdots+t_{r}\right) \right) -f\left( x\right)\right] dt_{1}\cdots dt_{r}\right] \phi _{\sigma }\left( nx-k\right) }{\sum\limits_{k=\left\lceil na\right\rceil}^{\left\lfloor nb\right\rfloor -r}		\phi _{\sigma }\left( nx-k\right) } \\
	= &\left( \sum_{\substack{k=\left\lceil na\right\rceil \\\left\vert nx-k\right\vert \leq \delta n/2}}^{\left\lfloor nb\right\rfloor -r} + \sum_{\substack{k=\left\lceil na\right\rceil\\\left\vert nx-k\right\vert >\delta n/2}}^{\left\lfloor nb\right\rfloor-r} \right)\\
	\times &\frac{\left[ n^{r}\int\limits_{0}^{1/n}\cdots\int\limits_{0}^{1/n}\left[ \sum\limits_{m=1}^{r} \left( -1\right) ^{1-m} \binom{r}{m}f\left( \frac{k}{n}+\frac{m}{r}\left(t_{1}+t_{2}+\cdots+t_{r}\right) \right) -f\left( x\right) \right] dt_{1}\cdots dt_{r}\right] \phi _{\sigma }\left( nx-k\right)}{ \sum\limits_{k=\left\lceil na\right\rceil}^{\left\lfloor nb\right\rfloor }\phi _{\sigma }\left( nx-k\right) } \\
	\leq &\frac{1}{\phi _{\sigma }\left( 1\right) }\left\{ T_{1}+T_{2}\right\} .
\end{align*}
If $\left\vert nx-k\right\vert \leq \delta n/2$, for every $t_{i}\in \left[ 0,1/n\right] $, $i=1,\cdots,r$, we can write
\begin{equation*}
\left\vert \frac{k}{n}+\frac{1}{r}\left( t_{1}+t_{2}+\cdots+t_{r}\right)
-x\right\vert \leq \left\vert \frac{k}{n}-x\right\vert +\frac{1}{r}\left(
t_{1}+t_{2}+\cdots+t_{r}\right) \leq \frac{\delta }{2}+\frac{1}{n}\leq \delta
\end{equation*}
for $n>0$ sufficently large $n$. Also
\begin{align*}
&\left\vert \frac{k}{n}+\frac{m}{r}\left( t_{1}+t_{2}+\cdots+t_{r}\right) -%
\left[ \frac{k}{n}+\frac{m-1}{r}\left( t_{1}+t_{2}+\cdots+t_{r}\right) \right]
\right\vert \\
=&\left\vert \frac{m}{r}\left( t_{1}+t_{2}+\cdots+t_{r}\right) -\frac{m-1}{r}%
\left( t_{1}+t_{2}+\cdots+t_{r}\right) \right\vert \\
=&\frac{1}{r}\left( t_{1}+t_{2}+\cdots+t_{r}\right) \leq \frac{1}{n}\leq \delta
\end{align*}
for $n>0$ sufficently large, for every $m=2,\cdots,r$. Thus, by direct calculation we can write (for details see \cite{Costarelli2024})
\begin{align*}
&\sum_{m=1}^{r}\left( -1\right) ^{1-m}\binom{r}{m}f\left( \frac{k}{n}+\frac{m\left( t_{1}+t_{2}+\cdots+t_{r}\right) }{r}\right) -f\left( x\right) \\
=&\sum_{m=1}^{r-2}\left( -1\right) ^{-m}\binom{r-1}{m}\left[ f\left( \frac{k}{n}+\frac{\left( m+1\right) \left( t_{1}+\cdots+t_{r}\right) }{r}\right)-f\left( \frac{k}{n}+\frac{m\left( t_{1}+\cdots+t_{r}\right) }{r}\right) \right]\\
+&\left( -1\right) ^{1-r}\left[ f\left( \frac{k}{n}+\left(t_{1}+\cdots+t_{r}\right) \right) -f\left( \frac{k}{n}+\frac{\left( r-1\right)\left( t_{1}+\cdots+t_{r}\right) }{r}\right) \right] \\
+&f\left( \frac{k}{n}+\frac{\left( t_{1}+t_{2}+\cdots+t_{r}\right) }{r}\right)-f\left( x\right).
\end{align*}
Let us first consider $T_{1}$.
\begin{align*}
\left\vert T_{1}\right\vert =&\left\vert \sum_{\substack{ k=\left\lceil na\right\rceil  \\ \left\vert nx-k\right\vert \leq \gamma n/2}}^{k=\left\lfloor nb\right\rfloor -r}\left\{ n^{r}\int_{0}^{1/n}\cdots\int_{0}^{1/n}\right.\right.\\
&\left.\left.\left[ \sum_{m=1}^{r-2}\left( -1\right) ^{-m}\binom{r-1}{m} \left[ f\left( \frac{k}{n}+\frac{\left( m+1\right) \left(t_{1}+t_{2}+\cdots+t_{r}\right) }{r}\right)-f\left( \frac{k}{n}+\frac{m\left( t_{1}+t_{2}+\cdots+t_{r}\right) }{r}\right)\right] \right.\right.\right.\\
+&\left.\left.\left.\left( -1\right) ^{1-r}\left[ f\left( \frac{k}{n}+\left( t_{1}+t_{2}+\cdots+t_{r}\right) \right)-f\left( \frac{k}{n}+\frac{\left( r-1\right) \left(t_{1}+t_{2}+\cdots+t_{r}\right) }{r}\right)\right]\right.\right.\right.\\
+&\left.\left.\left. f\left( \frac{k}{n}+\frac{\left( t_{1}+t_{2}+\cdots+t_{r}\right) }{r}\right)-f\left( x\right)\right] dt_{1}\cdots dt_{r}\right\} \phi _{\sigma }\left( nx-k\right) \right\vert\\
\leq &\sum_{\substack{ k=\left\lceil na\right\rceil  \\ \left\vert nx-k\right\vert \leq \gamma n/2}}^{k=\left\lfloor nb\right\rfloor -r}\left\{ n^{r}\int_{0}^{1/n}\cdots\int_{0}^{1/n} \right.\\
&\left. \sum_{m=1}^{r-2}\binom{r-1}{m}\left\vert f\left( \frac{k}{n}+\frac{\left( m+1\right) \left( t_{1}+\cdots+t_{r}\right) }{r}\right) -f\left( \frac{k}{n}+\frac{m\left( t_{1}+\cdots+t_{r}\right) }{r}\right)\right\vert \right. \\
+&\left. \left\vert f\left( \frac{k}{n}+\left( t_{1}+\cdots+t_{r}\right) \right) -f\left( \frac{k}{n}+\frac{\left( r-1\right) \left( t_{1}+\cdots+t_{r}\right) }{r}\right)\right\vert \right.\\
+&\left.\left\vert f\left( \frac{k}{n}+\frac{\left( t_{1}+\cdots+t_{r}\right) }{r}\right) -f\left( x\right) \right\vert dt_{1}\cdots dt_{r}\right\} \phi _{\sigma }\left( nx-k\right)\\
\leq &2^{r-1}\varepsilon \sum_{\substack{ k=\left\lceil na\right\rceil  \\ \left\vert nx-k\right\vert \leq \gamma n/2}}^{k=\left\lfloor nb\right\rfloor-r}\phi _{\sigma }\left( nx-k\right)\leq 2^{r-1}\varepsilon. 
\end{align*}
Now, we estimate $T_{2}.$ By Lemma \ref{lemma-phi-prop} (viii), we have
\begin{align*}
&\left\vert T_{2}\right\vert \\
=&\left\vert \sum_{\substack{ k=\left\lceil na\right\rceil  \\ \left\vert nx-k\right\vert >\gamma n/2}}^{k=\left\lfloor nb\right\rfloor-r }\left\{ n^{r}\int_{0}^{1/n}\cdots\int_{0}^{1/n}\left[  \sum_{m=1}^{r}\left( -1\right) ^{1-m}\binom{r}{m} f\left( \frac{k}{n}+\frac{m}{r}\left( t_{1}+\cdots+t_{r}\right) \right) -f\left(x\right)\right] dt_{1}\cdots dt_{r}\right\} \phi _{\sigma }\left( nx-k\right) \right\vert\\
\leq &\sum_{\substack{ k=\left\lceil na\right\rceil  \\ \left\vert
nx-k\right\vert >\gamma n/2}}^{k=\left\lfloor nb\right\rfloor -r}\left\{
n^{r}\int_{0}^{1/n}\cdots\int_{0}^{1/n}\left[ \sum_{m=1}^{r} \binom{r}{m} \left\vert f\left( \frac{k}{n}+\frac{m}{r}\left( t_{1}+\cdots+t_{r}\right)\right) -f\left( x\right) \right\vert\right] dt_{1}\cdots dt_{r}\right\} \phi _{\sigma }\left( nx-k\right) \\
\leq &2\left\Vert f\right\Vert _{\infty }\left( 2^{r}-1\right) \sum_{\substack{ k=\left\lceil na\right\rceil  \\ \left\vert nx-k\right\vert>\gamma n/2}}^{k=\left\lfloor nb\right\rfloor-r }\phi _{\sigma }\left(nx-k\right) \leq 2\left\Vert f\right\Vert _{\infty }\left( 2^{r}-1\right)\varepsilon.
\end{align*}
Hence, combining the estimates $T_{1}$ and $T_{2}$ we get
\begin{equation*}
\left( F_{n}^{r}f\right) \left( x\right) -f\left( x\right) \leq \frac{1}{\phi _{\sigma }\left( 1\right) }\left\{ 2^{r-1}\varepsilon+2\left\Vert f\right\Vert _{\infty }\left( 2^{r}-1\right) \varepsilon
\right\}
\end{equation*}
which completes the first part of proof. The second part of the theorem
follows replacing the parameter $\delta >0$ of the continuity \ of the $f$
with the corresponding one of the uniform continuity of $f$ and all the
above estimates hold uniformly with respect to $x\in \left[ a,b\right] $.
The completes the proof of the second part of the theorem.
\end{proof}

Now, we will state the quantitative estimate theorem for functions belonging to $f\in C\left([a, b]\right)$ via modulus of continuity defined by 
\[ \omega\left(f ; \delta\right) := \sup_{\substack{t, x\in \mathbb{R}\\\left|t-x\right|<\delta}}\left|f\left(t\right) - f\left(x\right)\right| \]
for $\delta>0$ and $f\in C\left([a, b]\right)$. It is well known that, for every $f\in CB\left(\mathbb{R}\right)$ and any $\lambda>0$ 
\[ \omega\left(f;\lambda\delta\right) \leq \left(1+\lambda\right) \omega\left(f;\delta\right) \]
holds.

\begin{theorem}
	Let $r\in \mathbb{N}^{+}, r>1, n\in \mathbb{N}$, such that $\lceil n a\rceil\leq{\lfloor n b\rfloor-r}$ and $f\in C\left([a,b]\right)$. Then,
	\begin{align*}
		\left|\left(F_{n}^{r} f\right)\left(x\right) - f\left(x\right)\right| \leq \dfrac{\omega\left(f;n^{-1}\right)}{\phi_{\sigma}\left(r+1\right)}\left(1 + M_{1}\left(\phi_{\sigma}\right)\right).
	\end{align*}
\end{theorem}

\begin{proof}
	By simple calculation, one can see the identity
	\begin{align*}
		\sum_{m=1}^{r} \left(-1\right)^{1-m} {r\choose m} \dfrac{m}{r} = -\dfrac{1}{r} \sum_{m=1}^{r} \left(-1\right)^{m} {r\choose m} m &= -\dfrac{1}{r} \sum_{m=1}^{r} \left(-1\right)^{m} r {r-1\choose m-1}\\
		&= - \sum_{m=1}^{r} \left(-1\right)^{m} {r-1\choose m-1}\\
		&= - \sum_{m=0}^{r-1} \left(-1\right)^{m+1} {r-1\choose m}\\
		&= \sum_{m=0}^{r-1} \left(-1\right)^{m} {r-1\choose m}\\
		&= \left(-1 + 1\right)^{r-1} =\begin{cases}
			1,& \text{ if } r=1\\
			0,& \text{ if } r>1.
		\end{cases}
	\end{align*}
	Now, using the property of modulus of continuity and definition of the operators we have
	\begin{align*}
		&\left|\left(F_{n}^{r} f\right)\left(x\right) - f\left(x\right)\right| \\
		\leq& \dfrac{1}{\phi_{\sigma}\left(r+1\right)} \sum_{k=\lceil na \rceil}^{\lfloor nb \rfloor -r} \left[n^{r} \int_{0}^{1/n} \cdots \int_{0}^{1/n} \left|\sum_{m=1}^{r} \left(-1\right)^{1-m} {r\choose m} \left(\left|f\left(\dfrac{k}{n} + \dfrac{m}{r}\left(t_{1} + \cdots + t_{r}\right) \right) - f\left(x\right)\right|\right)\right|d_{t_{1}}\dots d_{t_{r}}\right]\left|\phi_{\sigma}\left(nx-k\right)\right|\\
		\leq& \dfrac{1}{\phi_{\sigma}\left(r+1\right)} \sum_{k=\lceil na \rceil}^{\lfloor nb \rfloor -r} \left[ n^{r} \int_{0}^{1/n} \cdots \int_{0}^{1/n} \left|\sum_{m=1}^{r} \left(-1\right)^{1-m}{r\choose m} \left(\omega\left(f;\left|\dfrac{k}{n} + \dfrac{m}{r} \left(t_{1} + \cdots + t_{r}\right) - x\right|\right)\right)\right|d_{t_{1}}\dots d_{t_{r}}\right]\left|\phi_{\sigma}\left(nx-k\right)\right|\\
		\leq& \dfrac{\omega\left(f;\delta\right)}{\phi_{\sigma}\left(r+1\right)} \sum_{k=\lceil na \rceil}^{\lfloor nb \rfloor -r} \left[ n^{r} \int_{0}^{1/n} \cdots \int_{0}^{1/n} \left|\sum_{m=1}^{r} \left(-1\right)^{1-m}{r\choose m} \left(1 + \dfrac{1}{\delta} \left|\dfrac{k}{n} + \dfrac{m}{r}\left(t_{1} + \cdots + t_{r}\right)  - x\right|\right) \right|d_{t_{1}}\dots d_{t_{r}}\right]\left|\phi_{\sigma}\left(nx-k\right)\right|\\
		\leq& \dfrac{\omega\left(f;\delta\right)}{\phi_{\sigma}\left(r+1\right)} \sum_{k=\lceil na \rceil}^{\lfloor nb \rfloor -r} \left[ n^{r} \int_{0}^{1/n} \cdots \int_{0}^{1/n} \right.\\
		\times&\left. \left|\sum_{m=1}^{r} \left(-1\right)^{1-m}{r\choose m}\right| + \dfrac{1}{\delta}\left|\dfrac{k}{n} - x\right|\left|\sum_{m=1}^{r} \left(-1\right)^{1-m}{r\choose m} \right| + \left|t_{1} + \cdots + t_{r}\right| \left|\sum_{m=1}^{r} \left(-1\right)^{1-m}{r\choose m} \dfrac{m}{r}\right|d_{t_{1}}\dots d_{t_{r}} \right]\left|\phi_{\sigma}\left(nx-k\right)\right|\\
		\leq& \dfrac{\omega\left(f;\delta\right)}{\phi_{\sigma}\left(r+1\right)} \sum_{k=\lceil na \rceil}^{\lfloor nb \rfloor -r} \left[ 1 + \dfrac{1}{\delta n} \left|k - nx\right|\right]\left|\phi_{\sigma}\left(nx-k\right)\right| \\
		\leq& \dfrac{\omega\left(f;\delta\right)}{\phi_{\sigma}\left(r+1\right)} \left[1 + \dfrac{1}{\delta n}M_{1}\left(\phi_{\sigma}\right)\right].
	\end{align*}
	If we choose $\delta=n^{-1}$, we conclude
	\begin{equation*}
		\left|\left(F_{n}^{r} f\right)\left(x\right) - f\left(x\right)\right| \leq \dfrac{\omega\left(f;n^{-1}\right)}{\phi_{\sigma}\left(r+1\right)}\left(1 + M_{1}\left(\phi_{\sigma}\right)\right)
	\end{equation*}
	which is desired.
\end{proof}

\end{document}